\documentclass[11pt]{article}
\usepackage{amsmath}    
\usepackage{amsthm}    
\usepackage{amssymb}   
\usepackage{geometry}  
\usepackage{hyperref} 
\usepackage{mathrsfs}  
\usepackage{enumitem}  
\usepackage{xcolor}
\usepackage{graphicx}
\usepackage{tikz}
\usepackage{subcaption}
\usepackage{asymptote}
\usetikzlibrary{positioning,arrows.meta,fit,calc,shapes.multipart}
\usepackage{array,makecell,booktabs}
\newcolumntype{L}[1]{>{\raggedright\arraybackslash}p{#1}}
\newcolumntype{C}[1]{>{\centering\arraybackslash}p{#1}}
\newcommand{\thickhline}{\specialrule{1.5pt}{0pt}{0pt}}  
\newcommand{\medhline}{\specialrule{0.8pt}{0pt}{0pt}}    
\renewcommand{\footnotesize}{\tiny}

\newtheorem{theorem}{Theorem}[section]
\newtheorem{lemma}[theorem]{Lemma}
\newtheorem{proposition}[theorem]{Proposition}
\newtheorem{corollary}[theorem]{Corollary}
\newtheorem{definition}[theorem]{Definition}
\newtheorem{remark}[theorem]{Remark}

\newcommand{\R}{\mathbb{R}} 
 
\newcommand{\Z}{\mathbb{Z}} 
\newcommand{\Q}{\mathbb{Q}} 
\newcommand{\RQ}{\mathbb{R}\setminus\mathbb{Q}} 
\newcommand{\D}{\mathbb{D}}
\newcommand{\dirich}{\overline{\Omega}} 
\newcommand{\F}{F_4^*}
\newcommand{\T}{\mathrm{T}}
\newcommand{\A}{\mathcal{A}}
\newcommand{\C}{\mathcal{C}}
\newcommand{\mubound}{\frac{19425+111\sqrt{26565}}{2030}} 

\begin{document}

\title{An Improved Upper Bound for the Dirichlet Spectrum in Diophantine Approximation}
\author{Zixuan Peng, Siyuan Wang, Ethan Wang}
\date{\today}
\maketitle

\begin{abstract}
We study the continuous part of the Dirichlet spectrum \(\mathbb{D}\) and improve the best previously published upper bound for the ray-origin constant $\delta$. Building on and refining Ivanov's approach from \cite{ivanov1980}, we introduce a Cantor-type set \(\F\) defined by certain restrictions on partial quotients. For its thickness, we prove \(\tau(\log(\F))>1\), and apply sum-set results for Cantor sets to prove that the set $\F\cdot \F$ is an interval. Finally, we establish a new upper bound
 $\delta\le \frac{111(397+\sqrt{26565})}{65522}\approx0.94866$.\end{abstract}

\section{Introduction}
\label{sec:intro}
We recall standard notation and results from the theory of continued fractions and the Dirichlet spectrum.
For $\alpha \in \R \setminus \Q$, consider the \emph{irrationality measure function}
\[
\psi_\alpha(t) = \min_{1 \leq q \leq t} \|q\alpha\|, \quad \|x\| = \min\{\{x\}, 1-\{x\}\}.
\]
The \emph{Dirichlet constant} for an irrational number $\alpha$ is defined as
\[
d(\alpha) = \limsup_{t \to \infty} t \psi_\alpha(t).
\]
The set
\[
\mathbb{D} = \{ d = d(\alpha),\,\,\, \alpha \in \mathbb{R}\}
\]
of all values of $d(\alpha)$ is known as the Dirichlet spectrum.

\vskip+0.3cm
We give a brief description of the main results dealing with the structure of $\mathbb{D}$.
It is well known that
$
\psi_\alpha(t) \leq \frac{1}{t},\ \forall t\in\R \text{ and so } \D \subset [0, 1].
$
In 1937, G. Szekeres showed in \cite{szekeres} that
$$\mathbb{D}\subset \left[\frac{5+\sqrt{5}}{10},1\right].$$
Morimoto  \cite{morim}  discovered a discrete
subset 
$$D_k=d\left(\alpha_k\right)=\frac{2\alpha_k+1}{2\alpha_k+2},\text{ where } \alpha_0=[\overline{1}], \alpha_k=[\overline{\underbrace{1;1,\dots,1}_{2k-1},2}]$$
of minimal values from Dirichlet spectrum $\mathbb{D}$.
Lesca in his thesis \cite{lesca} gave a complete proof of the equality 
$$\mathbb{D}\cap\left[\frac{5+\sqrt{5}}{10},\frac{1+\sqrt{5}}{4}\right)=\{D_j\mid j\in\mathbb{Z}_{>0}\}.$$
 \\
 As for the structure of $\mathbb{D}$ close to 1,
in \cite{DN} Diviš and Novák proved the following result.
\begin{theorem}
There exists $d^{*} < 1$ such that $[d^*, 1] \subset \mathbb{D}$.
\end{theorem}
$\text{}$\\
 The purpose of the present paper is to study the value
$$
\delta = \inf\{d\in \mathbb{R}: \,\,\ (d,1]\subset \mathbb{D}\}.
$$
The reader may find in \cite{dirich,p} a more detailed overview of the current progress regarding the Dirichlet spectrum.

\vskip+0.3cm

The basic tool to study the Dirichlet spectrum is the machinery of continued fractions.
Let $\alpha=[x_0;x_1,x_2,\dots]\in \RQ$ be represented as an ordinary continued fraction.
Define
$$\alpha_n = [x_n; x_{n+1}, x_{n+2}, \dots]\,\,\,\,\text{and}
\,\,\,\,
\alpha_n^* = [0; x_n, x_{n-1}, \dots, x_0].$$
Then
$$\psi_\alpha(t)=\frac{1}{q_{n+1}\left(1+\frac{\alpha^*_{n+1}}{\alpha_{n+2}}\right)}\,\,\,\,\,
\,\,\,\,\,\,\forall q_n\le t< q_{n+1}.
$$
The last equality is known as Perron's formula. From this formula we conclude that 
\[
d(\alpha)=\limsup_{\nu\to\infty}\frac{1}{1+\dfrac{\alpha_{\nu+1}^*}{\alpha_{\nu+2}}}
=\frac{1}{1+\dfrac{1}{\displaystyle\limsup_{\nu\to\infty}\dfrac{\alpha_{\nu+1}}{\alpha_{\nu}^*}}}
= \frac{1}{1+\dfrac{1}{\rho(\alpha)}},
\]
where
\[
\rho(\alpha):=\limsup_{\nu\to\infty}\frac{\alpha_{\nu+1}}{\alpha_{\nu}^*}= \limsup_{n \to \infty} \left([x_n; x_{n-1}, \ldots, x_0] \cdot [x_{n+1}; x_{n+2}, \ldots]\right).
\]
Let
$$\dirich = \{ \rho(\alpha) \mid \alpha \in \RQ \}
$$
and
$$
\mu = \inf \{ z \in \R \mid [z, +\infty) \subset \dirich \}.$$
Since $\delta=\frac{\mu}{1+\mu},$ it suffices to bound $\mu$ to obtain a bound on $\delta$.\\

In \cite{ivanov1980}
Ivanov proved the following result.
\begin{theorem}
$
\frac{5 + 3\sqrt{5}}{2} \approx 5.8541\ldots \leq \mu \leq 10 + 6\sqrt{2} \approx 18.4852\ldots
$
\end{theorem}

Now we formulate our main result and its corollary which improve on the result by Ivanov.
\begin{theorem}[Main Theorem]
\label{thm:main}
The constant $\mu$ satisfies the inequality
\[
\mu \leq \mubound\approx18.4811 .
\]
\end{theorem}
\begin{corollary}
    $\delta\le \frac{111(397+\sqrt{26565})}{65522}\approx0.94866 .$
\end{corollary}

Our proofs develop the approach from \cite{ivanov1980} and are related to a deeper analysis of the corresponding Cantor-type sets of continued fractions with bounded partial quotients.

\section{Cantor Sets and Thickness}

We begin with the following.

\begin{definition}
A \textbf{Cantor-type set} $\mathcal{E}$ on a closed interval $I$ is defined as
\begin{equation}\label{1q}
\mathcal{E} = I \setminus \bigcup_{\alpha \in \mathbb{Z}_{\ge 0}} U_\alpha,
\end{equation}
where $\{U_{\alpha}\}_{\alpha\in \mathbb{Z}_{\geq 0}}$ is a countable set of mutually disjoint open intervals representing the ``gaps" removed from the original interval.
\end{definition}

Before analyzing these sets further, we establish a notation for their thickness. We will use $|A|=\sup A-\inf A$ to denote the length of an interval $A$.

A crucial property of any Cantor-type set is its ``thickness", which quantifies how much of the solid interval remains compared to the size of the gaps being removed. Intuitively, a thicker Cantor-type set has relatively small holes. 

\begin{definition}
Let the set $\mathcal{E} $ be of the form (\ref{1q}) and the order of the  intervals $U_\alpha$ be fixed.
For any removed gap $U_k$, let $I_k$ be the largest closed interval in
\[
\mathcal{E}_k = I \setminus \bigcup_{i=1}^{k-1} U_i
\]
that contains $U_k$. Denote by $I_k^{1}$ and $I_k^{2}$ the two neighbouring maximal subsegments of $I_k \setminus U_k$.\\
The thickness $\tau(\mathcal{E})$ is said to be at least $\tau$ if the inequality 
\[
\tau \cdot |U_k| \le \min\left\{ |I_k^{1}|, |I_k^{2}|\right\}
\]
holds for all $k$.
\end{definition}
\text{}\\
In the following, let $A+B=\{a+b\mid a\in A,\;b\in B\}$, $A\cdot B = \{ab\mid a\in A,\; b\in B\}$, and $\log A= \{ \log x \mid x \in A\}$, for sets $A,B\subseteq \R$.

\begin{theorem}[Hall Jr., Theorem 2.2 in~\cite{hall1947}]
Let $\mathcal{E}_1,\mathcal{E}_2$ be 2 Cantor-type sets defined on segments A and B such that:
$\tau(\mathcal{E}_1),\tau(\mathcal{E}_2) \geq 1, $
and $\frac{1}{3}\le\frac{|A|}{|B|}\le 3.$
Then $\mathcal{E}_1 + \mathcal{E}_2=A+B$.
\end{theorem}

\begin{corollary}
\label{cor:logset}
For any Cantor-type set $\mathcal{E}$ defined on a segment A with 
$\tau(\log\mathcal{E}) > 1$, the product $\mathcal{E} \cdot \mathcal{E}$ is just the segment  $A \cdot A$.  
\end{corollary}
The core of our proof relies on a Cantor-type set defined as
\[
\F = \left\{ x = [x_0; x_1, x_2, \dots] \;\middle|\,
\begin{aligned}
&(x_0,x_1)=(4,3),\\
&x_i \in \{1,2,3,4\}, \\
&(x_i, x_{i+1}) \neq (4,4), \\
&(x_i, x_{i+1}, x_{i+2},x_{i+3},x_{i+4}) \neq (4,1,4,1,4)
\end{aligned}
\right\}.
\]
\\
We use $\T[a_0,a_1,\dots,a_n]$ to denote the (potentially empty) minimal closed interval containing all continued fractions in $\F$ whose expansions begin with the prefix $a_0,a_1,\dots,a_n$, and we denote by $\T[a_0,a_1,\dots,a_n,_a^b]( a,b\in\mathbb{Z}_{>0})$ the interval $\left[\min\T[a_0,\dots,a_n,a],\max \T[a_0,\dots,a_n,b]\right]$. \\
Let 
\[
C_n=\bigcup_{a_0,\dots,a_n\in\{1,2,3,4\}} \T[a_0,\dots,a_n],\ \forall n\in\Z_{\geq0}.
\]
We will use the following properties of the sets $\T[a_0,a_1,\dots,a_n]$.
\begin{itemize}
    \item \text{Disjointness:}
    \[
    \T[\alpha_0,\alpha_1,\dots,\alpha_n]\cap \T[\beta_0,\beta_1,\dots,\beta_n]=\varnothing \,\;\text{if}\;\, (\alpha_0,\alpha_1,\dots,\alpha_n)\neq(\beta_0,\beta_1,\dots,\beta_n).
    \]
    \item \text{Nestedness:} $C_{n+1} \subset C_n$ for all $n$.  
\end{itemize}
As a consequence, we see that $\lim_{n\to\infty} C_n$ exists. Let this limit be denoted by $\mathcal{C}$.

\begin{proposition}
    $
    \F=\lim_{n\to\infty} C_n=\C.
    $
    \end{proposition}
\begin{proof}
\leavevmode
\\
    For every $n\in\Z_{\geq0}$ and every $x=[a_0; a_1, a_2, \dots]\in\F,\text{ one has }x\in \T[a_0,\dots,a_n]\subseteq C_n$. Hence $\F \subseteq \C$, and it remains to prove the reverse inclusion. Assume by contradiction that there exists $x=[a_0; a_1, a_2, \dots] \in \C\setminus\F$.
    \\
    For every $x\notin\F$, define
    the integer value 
    \[K(x) := \min \left\{ i \in \mathbb{N}\cup\{0\} \,\middle|\,
    \begin{aligned}
    &a_i \notin \{1,2,3,4\}, \quad \text{or} \\
    &(a_i, a_{i+1}) = (4,4), \quad \text{or} \\
    &(a_i, a_{i+1}, a_{i+2}, a_{i+3}, a_{i+4}) = (4,1,4,1,4)
    \end{aligned}\right\},
    \]
      which is well-defined and finite for every $x\notin\F$. However, $x\in \C$, so for arbitrary $i$ there is a unique $\T[\nu_0,\nu_1,\dots,\nu_i]\subseteq C_i$ which contains $x$.
    In particular, there exists a unique interval $\T[\nu_0,\nu_1,\dots,\nu_{K(x)+4}]$ such that $x\in\T[\nu_0,\nu_1,\dots,\nu_{K(x)+4}]$.
    Since $x=[a_0; a_1, a_2, \dots,a_{K(x)+4},\dots]$, by definition $$x\in\T[a_0,\dots,a_{K(x)},a_{K(x)+1},a_{K(x)+2},a_{K(x)+3},a_{K(x)+4}].$$However, $$\T[a_0,\dots,a_{K(x)},a_{K(x)+1},a_{K(x)+2},a_{K(x)+3},a_{K(x)+4}]=\varnothing$$ as no element in $\F$ begins with $(a_0,\dots,a_{K(x)},a_{K(x)+1},a_{K(x)+2},a_{K(x)+3},a_{K(x)+4})$. This contradiction completes the proof.
\end{proof}
\begin{theorem}
\label{thm:F4}
$\text{}$\\
(1) $\F$ is a Cantor-type set.\\
(2) $\tau(\F)\geq \frac{83497\sqrt{26565}-228339}{13158329} \approx 1.0169>1$.
\end{theorem}
\begin{proof}[Proof of \autoref{thm:F4}]
\leavevmode
\\
The proof proceeds as follows. We begin by defining several types of segments that serve as the structural basis for our construction. Using these types, we implement an iterative binary partitioning algorithm to generate a Cantor set, $\A$. We then demonstrate that $\A$ is equivalent to both $\C$ and $\F$ (Lemma \ref{lem:inclusion}). Finally, we utilize the previously defined segment categories to compute the thickness of $\A$, proving it satisfies the theorem's requirements (Lemma \ref{lemma:comptrick}, Corollary \ref{cor:F44}).

Define 
\[
\left[a_0;a_1,\dots,a_n,_{\alpha}^{\beta}\right]
=
\bigl[\, [a_0;a_1,\dots,a_n,\alpha]\ ,[a_0;a_1,\dots,a_n,\beta]\,\bigr],
\]
where $\alpha$ and $\beta$ are blocks of continued fractions.\\
\begin{remark}
Note that the relative order of the numbers $[a_0,a_1,\dots,a_n,\alpha]$ and $[a_0,a_1,\dots,a_n,\beta]$ depends on the parity of $n$. Here and below, we assume that $n$ is odd. If $n$ is even, one only needs to exchange the order of the intervals.
\end{remark}
$\text{}$\\
We define the following types of segments. They are all of the form $$\left[a_0;a_1,\dots,a_n,_{\alpha}^{\beta}\right],$$
where $\alpha$ and $\beta$ are strings given in the following table.\\

Throughout the following, we assume that $(a_0,\dots,a_n)$ is chosen from $\{4\}\times\{3\}\times \{1,2,3,4\}^{n-1}$. In the following table, we require the string inside $\T[\dots]$ to have no substring of the form $(4,4)$ or $(4,1,4,1,4)$.

\begin{center}
\renewcommand{\arraystretch}{1.5} 
\begin{tabular}{|C{2.5cm}|C{3.5cm}|C{2.5cm}|C{3.5cm}|C{3.9cm}|}
\thickhline
\multicolumn{5}{|c|}{\small\bfseries Segment Type Definition} \\ 
\thickhline
\makecell{$\text{Segment Type}$} &
\makecell{$\alpha$} &
\makecell{$\beta$} & 
\makecell{$\text{Representation }$\\$\text{using }\T[\dots]$}&
\makecell{$\text{Further}$\\$\text{Restrictions}$\\$ \text{ on }a_1,\dots, a_n$}\cr
\medhline

Type 1 & 
$\overline{1,4,1,4,1,3}$ & 
$\overline{4,1,4,1,3,1}$ & $\T[a_0,\dots,a_n]$ & $a_n \neq 4$\linebreak $(a_{n-1},a_n) \neq (4,1)$\linebreak $(a_{n-2},a_{n-1},a_n) \neq (4,1,4)$\linebreak $(a_{n-3},a_{n-2},a_{n-1},a_n)\neq (4,1,4,1)$\\ \hline

Type 2 & 
$\overline{1,4,1,4,1,3}$ & 
$\overline{3,1,4,1,4,1}$ & $\T[a_0,\dots,a_n,_{1}^{3}]$ & $a_n \neq 4$\linebreak $(a_{n-2},a_{n-1},a_n) \neq (4,1,4)$\linebreak $(a_{n-3},a_{n-2},a_{n-1},a_n)\neq (4,1,4,1)$ \\ \hline

Type 3 & 
$\overline{1,4,1,4,1,3}$ & 
$2,\overline{1,4,1,4,1,3}$ & $\T[a_0,\dots,a_n,_{1}^{2}]$ & $a_n \neq 4$\linebreak $(a_{n-2},a_{n-1},a_n) \neq (4,1,4)$\\ \hline

Type 4 & 
$4,\overline{3,1,4,1,4,1}$ & 
$\overline{4,1,4,1,3,1}$ & $\T[a_0,\dots,a_n,4]$ & $(a_{n-1},a_n) \neq (4,1)$ \\ \hline

Type 5 & 
$4,2,\overline{1,4,1,4,1,3}$ & 
$\overline{4,1,4,1,3,1}$ & $\T[a_0,\dots,a_n,4,_{1}^{2}]$ &$(a_{n-1},a_n) \neq (4,1)$ \\ \hline

Type 6 & 
$4,1,\overline{1,4,1,4,1,3}$ & 
$\overline{4,1,4,1,3,1}$ & $\T[a_0,\dots,a_n,4,1]$ & $(a_{n-1},a_n) \neq (4,1)$ \\ \hline

Type 7 & 
$4,1,4,\overline{3,1,4,1,4,1}$ & 
$\overline{4,1,4,1,3,1}$ & $\T[a_0,\dots,a_n,4,1,4,_{1}^{3}]$ & - \\ \hline

Type 8 & 
$4,1,4,2,\overline{1,4,1,4,1,3}$ & 
$\overline{4,1,4,1,3,1}$ & $\T[a_0,\dots,a_n,4,1,4,_{1}^{2}]$ & - \\ \hline

Type 9 & 
$4,1,4,1,\overline{1,4,1,4,1,3}$ & 
$\overline{4,1,4,1,3,1}$ & $\T[a_0,\dots,a_n,4,1,4,1]$ & - \\ \hline

\thickhline
\end{tabular}
\end{center}

To construct the Cantor-type set $\A$ mentioned in the outline of the proof, consider two families of segments $\{A_j^i\mid i\in \mathbb{Z}_{\ge 0},j\in\mathbb{Z}_{\geq1},1\le j\le 2^{i}\}$ and $\{G_j^i\mid i\in \mathbb{Z}_{\ge 1},j\in\mathbb{Z}_{\geq1},1\le j\le 2^{i-1}\}$ defined as follows. Let $A_1^0= \T[4,3]=\left[4;3,^{\overline{4,1,4,1,3,1}}_{\overline{1,4,1,4,1,3}}\right]$. For $i \geq 1$, we define $A_j^i$ for $j \in \{1, 2, \dots, 2^i\}$ inductively from $A_j^{i-1}\ \left(j \in \{1, 2, \dots, 2^{i-1}\}\right)$ according to the rules below and let $G_j^i=A_j^{i-1}\setminus(A^i_{2j-1}\cup A^i_{2j})$, which represent the gap between the segments.
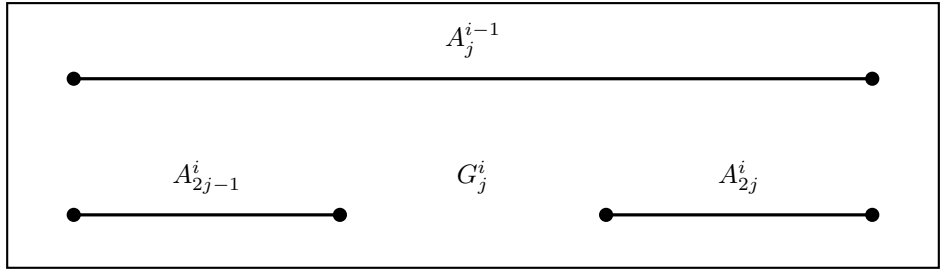
\begin{center}
    \begin{tikzpicture}[thick, x=0.11\textwidth, y=1cm,font=\footnotesize]
\tikzset{
    seg/.style = {line width=1.2pt, draw=black},
    dot/.style = {draw=black, fill=black, circle, inner sep=1.6pt},
    endlbl/.style = {below=3pt, font=\footnotesize, align=center},
    toplbl/.style = {above=4pt, font=\small\bfseries},
    arrowstyle/.style={-{Stealth[length=6pt,width=5pt]}, semithick},
    rulelab/.style = {font=\footnotesize, text=black, align=right}
  }
  
  \foreach \i in {0,...,7} \coordinate (x\i) at (\i,0);

  \coordinate (y0) at (0,0);      
  \coordinate (y1) at (0,-1.8);   
  
  \draw[seg] (x0 |- y0) -- (x6 |- y0);
  
  \node[dot] at (x0 |- y0) {};
  \node[dot] at (x6 |- y0) {};
  \node[toplbl] at ($0.5*(x0|-y0)+0.5*(x6|-y0)$) {$A_j^{i-1}$};

  \draw[seg] (x0 |- y1) -- (x2 |- y1);
  \draw[seg] (x4 |- y1) -- (x6 |- y1);
  
  \node[dot] at (x0 |- y1) {};
  \node[dot] at (x2 |- y1) {};
  \node[dot] at (x4 |- y1) {};
  \node[dot] at (x6 |- y1) {};
  \node[toplbl] at ($0.5*(x0|-y1)+0.5*(x6|-y1)$) {$G_{j}^{i}$};
  \node[toplbl] at ($0.5*(x0|-y1)+0.5*(x2|-y1)$) {$A_{2j-1}^i$};
  \node[toplbl] at ($0.5*(x4|-y1)+0.5*(x6|-y1)$) {$A_{2j}^i$};
  \draw[draw=black] (-0.5,1) rectangle ++(7,-3.5);
\end{tikzpicture}
\end{center}
We put \[\A_i:=\bigcup_{j=1}^{2^i} A_j^i,\qquad \A:=\bigcap_{i=0}^{\infty} \A_i.\]
Since for all $i\in\Z_{\geq0}$, $\A_{i+1}\subset A_i$, $\lim_{{i\to\infty}} A_i$ exists and is equal to $\A$. Note that intuitively, we can think of $\A$ as being constructed by starting with $A_1^0=T[4,3]$ and throwing out the segments $G_j^i$ at every step.\\
Note that the interval $A_1^0$ is an interval of Type 1. We now formalize the construction of the Cantor-type set $\A$. Starting from $A_1^0$, apply the corresponding rule from the 9 defined below to subdivide each segment type. (Note: Rule $i$ governs the subdivision of Type $i$ segments).\\
\begin{remark}
Again, the order of the intervals $A_{2j-1}^{i}$ and $A_{2j}^{i}$ depends on the parity of $n$. We assume $n$ is odd in all of the below. If $n$ is even, one only needs to exchange the order of $A_{2j-1}^{i}$ and $A_{2j}^{i}$.
\end{remark}
$\text{}$\\
We now define the \textbf{9 Rules} for subdividing the various interval types.
\begin{enumerate}

\item
The interval 
$
A_j^{i-1}=[a_0;a_1,\dots,a_n,_{\overline{1,4,1,4,1,3}}^{\overline{4,1,4,1,3,1}}]
$
is subdivided into
\[ 
A_{2j-1}^{i}=[a_0;a_1,\dots,a_n,_{\overline{1,4,1,4,1,3}}^{\overline{3,1,4,1,4,1}}]\text{ (Type 2)} \text{ and }
A_{2j}^{i}= [a_0;a_1,\dots,a_n,_{4,\overline{3,1,4,1,4,1}}^{\overline{4,1,4,1,3,1}}] \text{ (Type 4)}, 
\]
and this corresponds to the following diagram\footnote
{In the diagram, the interval to the left of the arrows is the parent interval (i.e. the interval to be divided, which is $A_j^{i-1}$), and the intervals to the right are the children intervals (i.e. the intervals which the parent is divided into, which are $A_{2j-1}^{i}$ and $A_{2j}^{i}$). The different formatting (continued fraction versus $T[*]$) can be derived simply from the definition of $T[*]$. For example in this case, $T[a_0,a_1,\dots a_n]=[a_0;a_1,\dots,a_n,_{\overline{1,4,1,4,1,3}}^{\overline{4,1,4,1,3,1}}]$.}.

\begin{center}
    \begin{tikzpicture}
  \def\hsep{30mm}

  \tikzset{
    box/.style = {draw, thick, rounded corners=3pt, inner sep=4pt, font=\small, align=center},
    arrow/.style = {-{Stealth[length=6pt,width=6pt]}, thick},
  }

  \node (P) at (0,0) {$\T[a_0,\dots,a_n]$};
  \node (C1) at ($(P.east)+(\hsep,0.8)$) {$\T[a_0,\dots,a_n,_1^3]$};
  \node (C2) at ($(P.east)+(\hsep,-0.8)$) {$\T[a_0,\dots,a_n,4]$};

  \draw[arrow] (P.east) to[out=0,in=180] (C1.west);
  \draw[arrow] (P.east) to[out=0,in=180] (C2.west);

  \node[draw, thick, fit=(P)(C1)(C2), inner sep=8pt]{};
\end{tikzpicture}
\end{center}
\item
The interval 
$
A_{j}^{i-1}=[a_0;a_1,\dots,a_n,_{\overline{1,4,1,4,1,3}}^{\overline{3,1,4,1,4,1}}]
$
is subdivided into
\[
A_{2j-1}^{i}=[a_0;a_1,\dots,a_n,_{\overline{1,4,1,4,1,3}}^{2,\overline{1,4,1,4,1,3}}]\text{ (Type 3)} \text{ and } 
A_{2j}^{i}=[a_0;a_1,\dots,a_n,_{3,\overline{4,1,4,1,3,1}}^{\overline{3,1,4,1,4,1}}]\footnote{This is of type 1 because $\overline{3,1,4,1,4,1}$ is the same as $3,\overline{1,4,1,4,1,3}$. Similar simplifications of continued fraction strings will occur below.}\text{ (Type 1)},
\]
which is represented by the following diagram:
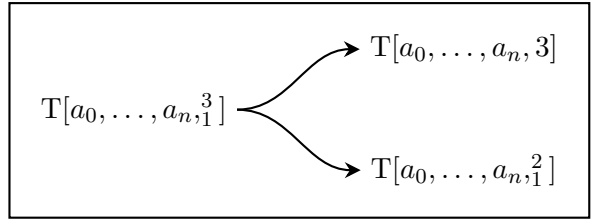
\begin{center}
    \begin{tikzpicture}
  \def\hsep{30mm}

  \tikzset{
    box/.style = {draw, thick, rounded corners=3pt, inner sep=4pt, font=\small, align=center},
    arrow/.style = {-{Stealth[length=6pt,width=6pt]}, thick},
  }

  \node (P) at (0,0) {$\T[a_0,\dots,a_n,_1^3]$};
  \node (C1) at ($(P.east)+(\hsep,0.8)$) {$\T[a_0,\dots,a_n,3]$};
  \node (C2) at ($(P.east)+(\hsep,-0.8)$) {$\T[a_0,\dots,a_n,_1^2]$};

  \draw[arrow] (P.east) to[out=0,in=180] (C1.west);
  \draw[arrow] (P.east) to[out=0,in=180] (C2.west);

  \node[draw, thick, fit=(P)(C1)(C2), inner sep=8pt]{};
\end{tikzpicture}

\end{center}

\item
The interval 
$
A_{j}^{i-1}=[a_0;a_1,\dots,a_n,_{\overline{1,4,1,4,1,3}}^{2,\overline{1,4,1,4,1,3}}]
$
is subdivided into
\[
A_{2j-1}^{i}=[a_0;a_1,\dots,a_n,_{\overline{1,4,1,4,1,3}}^{1,\overline{1,4,1,4,1,3}}] \text{ (Type 1)} \text{ and } 
A_{2j}^{i}=[a_0;a_1,\dots,a_n,_{2,\overline{4,1,4,1,3,1}}^{2,\overline{1,4,1,4,1,3}}]\text{ (Type 1)},
\]
which is represented by the following diagram:
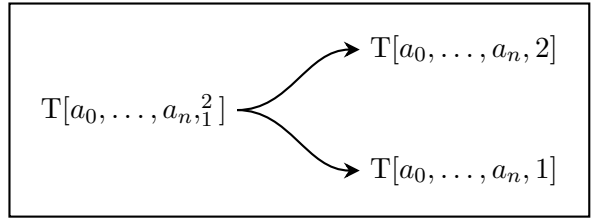
\begin{center}
    \begin{tikzpicture}
  \def\hsep{30mm}

  \tikzset{
    box/.style = {draw, thick, rounded corners=3pt, inner sep=4pt, font=\small, align=center},
    arrow/.style = {-{Stealth[length=6pt,width=6pt]}, thick},
  }

  \node (P) at (0,0) {$\T[a_0,\dots,a_n,_1^2]$};
  \node (C1) at ($(P.east)+(\hsep,0.8)$) {$\T[a_0,\dots,a_n,2]$};
  \node (C2) at ($(P.east)+(\hsep,-0.8)$) {$\T[a_0,\dots,a_n,1]$};

  \draw[arrow] (P.east) to[out=0,in=180] (C1.west);
  \draw[arrow] (P.east) to[out=0,in=180] (C2.west);

  \node[draw, thick, fit=(P)(C1)(C2), inner sep=8pt]{};
\end{tikzpicture}

\end{center}

\item
The interval 
$
A_{j}^{i-1}=[a_0;a_1,\dots,a_n,^{\overline{4,1,4,1,3,1}}_{4,\overline{3,1,4,1,4,1}}]
$
is subdivided into
\[
A_{2j-1}^{i}=[a_0;a_1,\dots,a_n,_{4,3,\overline{1,4,1,4,1,3}}^{4,3,\overline{4,1,4,1,3,1}}] \text{ (Type 1)} \text{ and } A_{2j}^{i}=[a_0;a_1,\dots,a_n,_{4,2,\overline{1,4,1,4,1,3}}^{\overline{4,1,4,1,3,1}}]\text{ (Type 5)},
\]
which is represented by the following diagram:
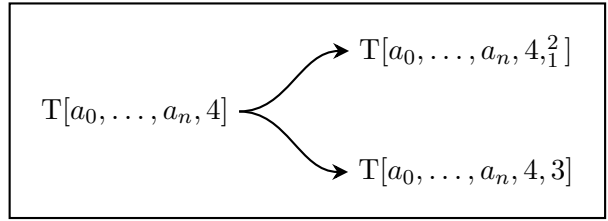
\begin{center}
    \begin{tikzpicture}
  \def\hsep{30mm}

  \tikzset{
    box/.style = {draw, thick, rounded corners=3pt, inner sep=4pt, font=\small, align=center},
    arrow/.style = {-{Stealth[length=6pt,width=6pt]}, thick},
  }

  \node (P) at (0,0) {$\T[a_0,\dots,a_n,4]$};
  \node (C1) at ($(P.east)+(\hsep,0.8)$) {$\T[a_0,\dots,a_n,4,_1^2]$};
  \node (C2) at ($(P.east)+(\hsep,-0.8)$) {$\T[a_0,\dots,a_n,4,3]$};

  \draw[arrow] (P.east) to[out=0,in=180] (C1.west);
  \draw[arrow] (P.east) to[out=0,in=180] (C2.west);

  \node[draw, thick, fit=(P)(C1)(C2), inner sep=8pt]{};
\end{tikzpicture}

\end{center}
\item
The interval 
$
A_{j}^{i-1}=[a_0;a_1,\dots,a_n,_{4,2,\overline{1,4,1,4,1,3}}^{\overline{4,1,4,1,3,1}}]
$
is subdivided into
\[
A_{2j-1}^{i}=[a_0;a_1,\dots,a_n,_{4,2,\overline{1,4,1,4,1,3}}^{4,2,\overline{4,1,4,1,3,1}}] \text{ (Type 1)} \text{ and } A_{2j}^{i}=[a_0;a_1,\dots,a_n,_{4,1,\overline{1,4,1,4,1,3}}^{4,1,\overline{4,1,3,1,4,1}}]\text{ (Type 6)},
\]
which is represented by the following diagram:
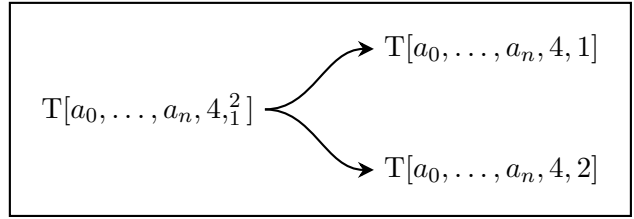
\begin{center}
    \begin{tikzpicture}
  \def\hsep{30mm}

  \tikzset{
    box/.style = {draw, thick, rounded corners=3pt, inner sep=4pt, font=\small, align=center},
    arrow/.style = {-{Stealth[length=6pt,width=6pt]}, thick},
  }

  \node (P) at (0,0) {$\T[a_0,\dots,a_n,4,_1^2]$};
  \node (C1) at ($(P.east)+(\hsep,0.8)$) {$\T[a_0,\dots,a_n,4,1]$};
  \node (C2) at ($(P.east)+(\hsep,-0.8)$) {$\T[a_0,\dots,a_n,4,2]$};

  \draw[arrow] (P.east) to[out=0,in=180] (C1.west);
  \draw[arrow] (P.east) to[out=0,in=180] (C2.west);

  \node[draw, thick, fit=(P)(C1)(C2), inner sep=8pt]{};
\end{tikzpicture}

\end{center}
\item
The interval 
$
A_{j}^{i-1}=[a_0;a_1,\dots,a_n,_{4,1,\overline{1,4,1,4,1,3}}^{\overline{4,1,4,1,3,1}}]
$
is subdivided into
\[
A_{2j-1}^{i}=[a_0;a_1,\dots,a_n,_{4,1,\overline{1,4,1,4,1,3}}^{4,1,\overline{3,1,4,1,4,1}}] \text{ (Type 2)} \text{ and } 
A_{2j}^{i}=[a_0;a_1,\dots,a_n,_{4,1,4,\overline{3,1,4,1,4,1}}^{4,1,\overline{4,1,3,1,4,1}}]\text{ (Type 7)},
\]
which is represented by the following diagram:
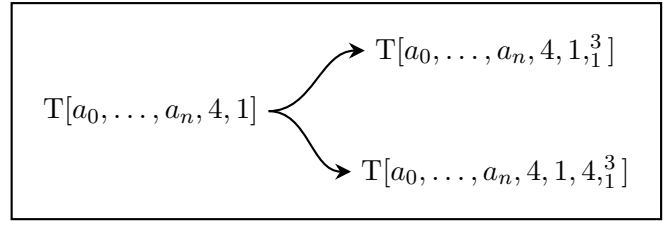
\begin{center}
    \begin{tikzpicture}
  \def\hsep{30mm}

  \tikzset{
    box/.style = {draw, thick, rounded corners=3pt, inner sep=4pt, font=\small, align=center},
    arrow/.style = {-{Stealth[length=6pt,width=6pt]}, thick},
  }

  \node (P) at (0,0) {$\T[a_0,\dots,a_n,4,1]$};
  \node (C1) at ($(P.east)+(\hsep,0.8)$) {$\T[a_0,\dots,a_n,4,1,_1^3]$};
  \node (C2) at ($(P.east)+(\hsep,-0.8)$) {$\T[a_0,\dots,a_n,4,1,4,_1^3]$};

  \draw[arrow] (P.east) to[out=0,in=180] (C1.west);
  \draw[arrow] (P.east) to[out=0,in=180] (C2.west);

  \node[draw, thick, fit=(P)(C1)(C2), inner sep=8pt]{};
\end{tikzpicture}

\end{center}

\item
The interval 
$
A_{j}^{i-1}=[a_0;a_1,\dots,a_n,_{4,1,4,\overline{3,1,4,1,4,1}}^{\overline{4,1,4,1,3,1}}]
$
is subdivided into
\[
A_{2j-1}^{i}=[a_0;a_1,\dots,a_n,_{4,1,4,\overline{3,1,4,1,4,1}}^{4,1,4,3,\overline{4,1,4,1,3,1}}] \text{ (Type 1)} \text{ and } A_{2j}^{i}=[a_0;a_1,\dots,a_n,_{4,1,4,2,\overline{1,4,1,4,1,3}}^{\overline{4,1,4,1,3,1}}]\text{ (Type 8)},
\]
which is represented by the following diagram:
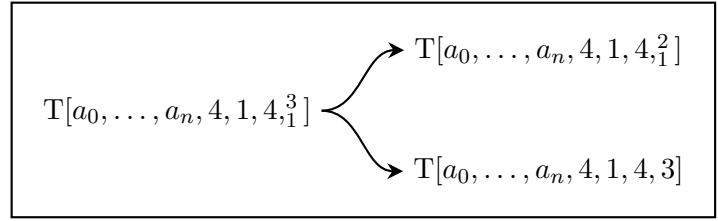
\begin{center}
    \begin{tikzpicture}
  \def\hsep{30mm}

  \tikzset{
    box/.style = {draw, thick, rounded corners=3pt, inner sep=4pt, font=\small, align=center},
    arrow/.style = {-{Stealth[length=6pt,width=6pt]}, thick},
  }

  \node (P) at (0,0) {$\T[a_0,\dots,a_n,4,1,4,_1^3]$};
  \node (C1) at ($(P.east)+(\hsep,0.8)$) {$\T[a_0,\dots,a_n,4,1,4,_1^2]$};
  \node (C2) at ($(P.east)+(\hsep,-0.8)$) {$\T[a_0,\dots,a_n,4,1,4,3]$};

  \draw[arrow] (P.east) to[out=0,in=180] (C1.west);
  \draw[arrow] (P.east) to[out=0,in=180] (C2.west);

  \node[draw, thick, fit=(P)(C1)(C2), inner sep=8pt]{};
\end{tikzpicture}

\end{center}
\item
The interval 
$
A_{j}^{i-1}=[a_0;a_1,\dots,a_n,_{4,1,4,2,\overline{1,4,1,4,1,3}}^{\overline{4,1,4,1,3,1}}]
$
is subdivided into
\[
A_{2j-1}^{i}=[a_0;a_1,\dots,a_n,_{4,1,4,2,\overline{1,4,1,4,1,3}}^{4,1,4,2,\overline{4,1,4,1,3,1}}] \text{ (Type 1)} \text{ and } A_{2j}^{i}=[a_0;a_1,\dots,a_n,_{4,1,4,1,\overline{1,4,1,4,1,3}}^{\overline{4,1,4,1,3,1}}]\text{ (Type 9)},
\]
which is represented by the following diagram:
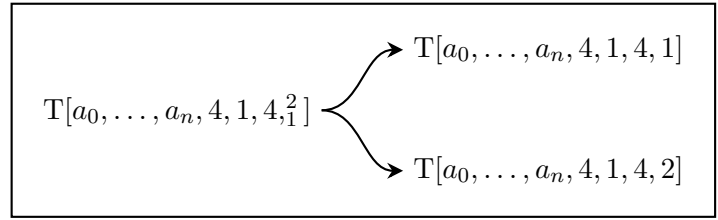
\begin{center}
    \begin{tikzpicture}
  \def\hsep{30mm}

  \tikzset{
    box/.style = {draw, thick, rounded corners=3pt, inner sep=4pt, font=\small, align=center},
    arrow/.style = {-{Stealth[length=6pt,width=6pt]}, thick},
  }

  \node (P) at (0,0) {$\T[a_0,\dots,a_n,4,1,4,_1^2]$};
  \node (C1) at ($(P.east)+(\hsep,0.8)$) {$\T[a_0,\dots,a_n,4,1,4,1]$};
  \node (C2) at ($(P.east)+(\hsep,-0.8)$) {$\T[a_0,\dots,a_n,4,1,4,2]$};

  \draw[arrow] (P.east) to[out=0,in=180] (C1.west);
  \draw[arrow] (P.east) to[out=0,in=180] (C2.west);

  \node[draw, thick, fit=(P)(C1)(C2), inner sep=8pt]{};
\end{tikzpicture}

\end{center}

\item
The interval 
$
A_j^{i-1}=[a_0;a_1,\dots,a_n,_{4,1,4,1,\overline{1,4,1,4,1,3}}^{\overline{4,1,4,1,3,1}}]
$
is subdivided into
\[
A_{2j-1}^i=[a_0;a_1,\dots,a_n,_{4,1,4,1,\overline{1,4,1,4,1,3}}^{4,1,4,1,2,\overline{1,4,1,4,1,3}}] \text{ (Type 3) and } 
A^i_{2j}=[a_0;a_1,\dots,a_n,_{4,1,4,1,3,\overline{4,1,4,1,3,1}}^{\overline{4,1,4,1,3,1}}]\text{ (Type 1)},
\]
which is represented by the following diagram:
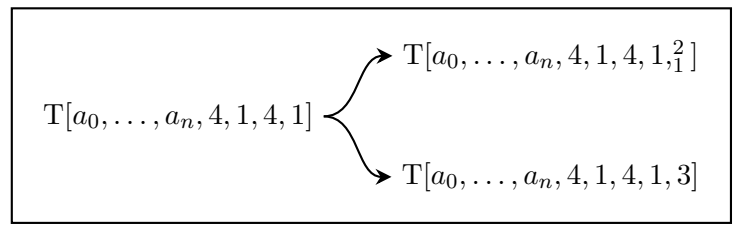
\begin{center}
    \begin{tikzpicture}
  \def\hsep{30mm}

  \tikzset{
    box/.style = {draw, thick, rounded corners=3pt, inner sep=4pt, font=\small, align=center},
    arrow/.style = {-{Stealth[length=6pt,width=6pt]}, thick},
  }

  \node (P) at (0,0) {$\T[a_0,\dots,a_n,4,1,4,1]$};
  \node (C1) at ($(P.east)+(\hsep,0.8)$) {$\T[a_0,\dots,a_n,4,1,4,1,_1^2]$};
  \node (C2) at ($(P.east)+(\hsep,-0.8)$) {$\T[a_0,\dots,a_n,4,1,4,1,3]$};

  \draw[arrow] (P.east) to[out=0,in=180] (C1.west);
  \draw[arrow] (P.east) to[out=0,in=180] (C2.west);

  \node[draw, thick, fit=(P)(C1)(C2), inner sep=8pt]{};
\end{tikzpicture}

\end{center}

\end{enumerate}
This subdivision process continues indefinitely, as summarized in the table below, each interval type is converted into two subintervals of specific types.
\vskip0.3cm
\begin{tabular}{|C{8.0cm}|C{4.0cm}|C{4.0cm}|}
\thickhline
\multicolumn{3}{|c|}{\Large\bfseries Segments Division List} \\  
\thickhline
\makecell{\textbf{Segment Type}} &
\makecell{\textbf{Child 1}\\\textbf{Type}} &
\makecell{\textbf{Child 2}\\\textbf{Type}} \\
\medhline

Type 1  & Type 2 & Type 4 \\ \hline
Type 2  & Type 3 & Type 1 \\ \hline
Type 3  & Type 1 & Type 1 \\ \hline
Type 4  & Type 1 & Type 5 \\ \hline
Type 5  & Type 1 & Type 6 \\ \hline
Type 6  & Type 2 & Type 7 \\ \hline
Type 7  & Type 1 & Type 8 \\ \hline
Type 8  & Type 1 & Type 9 \\ \hline
Type 9  & Type 3 & Type 1 \\ \hline

\thickhline
\end{tabular}
\begin{remark}
\label{rem:subdivision_rules}
\mbox{}\\
Starting from $\T[4,3]$, the aforementioned rules  divide $\T[a_0,a_1,\dots,a_n]$ into $\T[a_0,a_1,\dots,a_n,a_{n+1}]$ for every $a_{n+1}, (a_{n+1}=1,2,3,4)$ for which this is nonempty. We analyze every possible tail and illustrate their division process:\\
If $\T[a_0,a_1,\dots,a_n]$ is of the ``special" forms:\\
\[
\begin{aligned}
&\T[a_0,a_1,\dots,a_{n-1},4] \text{ (Type 4)},\text{ then apply rules 4, 5.} \\
&\T[a_0,a_1,\dots,a_{n-2},4,1]\text{ (Type 6)},\text{ then apply rules 6, 2, 3.} \\
&\T[a_0,a_1,\dots,a_{n-3},4,1,4]\text{ (Type 7)}\footnotemark,\text{ then apply rules 7, 8.}\\
&\T[a_0,a_1,\dots,a_{n-4},4,1,4,1]\text{ (Type 9)},\text{ then apply rules 9,3.} 
\end{aligned}
\]\footnotetext{The reason why $T[a_1,\dots,a_n,4,1,4]$ is of type 7 is because it is equal to $T[a_1,\dots,a_n,4,1,4,^{3}_{1}]$, as $T[a_1,\dots,a_n,4,1,4,4]$ is empty (by definition, there can't be consecutive 4's in the continued fraction of any number in $\F$).}\\
Otherwise, we apply Rules 1, 2, and 3 consecutively.\\
$\text{}$\\
For example, we consider the division of the initial segment $A_1^0=\T[4,3]$. We note that it is of Type $1$. By Rule $1$ above, we divide it into $\T[4,3,_{1}^{3}]$ of Type $2$ and $\T[4,3,4]$ of Type $4$. Then, by Rule $2$, we divide $\T[4,3,_{1}^{3}]$ into $\T[4,3,_{1}^{2}]$ of Type 3 and $\T[4,3,3]$ of Type $1$. Finally, by Rule $3$, we divide $\T[4,3,_{1}^{2}]$ into $\T[4,3,1]$ and $\T[4,3,2]$, both of Type $1$. We have thus split $\T[4,3]$ into:
$$\bigcup_{k=1}^{4}\T[4,3,k].$$
This is illustrated in the following diagram.
\begin{center}
\begin{tikzpicture}[thick, x=0.11\textwidth, y=1cm,font=\footnotesize]
 \tikzset{
    seg/.style = {line width=1.2pt, draw=black},
    dot/.style = {draw=black, fill=black, circle, inner sep=1.6pt},
    endlbl/.style = {below=3pt, font=\small, align=center},
    toplbl/.style = {above=4pt, font=\small\bfseries},
    arrowstyle/.style={-{Stealth[length=6pt,width=5pt]}, semithick},
    rulelab/.style = {font=\small, text=black, align=right}
  }
\draw[draw=black] (-1.5,1) rectangle ++(10,-7);

  \foreach \i in {0,...,7} \coordinate (x\i) at (1.16*\i,0);

  \coordinate (y0) at (0,0);      
  \coordinate (y1) at (0,-1.8);   
  \coordinate (y2) at (0,-3.6);   
  \coordinate (y3) at (0,-5.2);   

  \draw[seg] (x0 |- y0) -- (x7 |- y0);
  \node[dot] at (x0 |- y0) {};
  \node[dot] at (x7 |- y0) {};
  \node[toplbl] at ($0.5*(x0|-y0)+0.5*(x7|-y0)$) {$\T[4,3]$};

  \draw[seg] (x0 |- y1) -- (x5 |- y1);
  \node[dot] at (x0 |- y1) {};
  \node[dot] at (x6 |- y1) {};
  \node[toplbl] at ($0.5*(x0|-y1)+0.5*(x6|-y1)$) {$\T[4,3,_1^3]$};
  \draw[seg] (x6 |- y1) -- (x7 |- y1);
  \node[dot] at (x5 |- y1) {};
  \node[dot] at (x7 |- y1) {};
  \node[toplbl] at ($0.5*(x6|-y1)+0.5*(x7|-y1)$) {$\T[4,3,4]$};

  \draw[seg] (x0 |- y2) -- (x3 |- y2);
  \node[dot] at (x0 |- y2) {};
  \node[dot] at (x3 |- y2) {};
  \node[toplbl] at ($0.5*(x0|-y2)+0.5*(x3|-y2)$) {$\T[4,3,_1^2]$};

  \draw[seg] (x4 |- y2) -- (x5 |- y2);
  \node[dot] at (x4 |- y2) {};
  \node[dot] at (x5 |- y2) {};
  \node[toplbl] at ($0.5*(x4|-y2)+0.5*(x5|-y2)$) {$\T[4,3,3]$};

  \draw[seg] (x0 |- y3) -- (x1 |- y3);
  \node[dot] at (x0 |- y3) {};
  \node[dot] at (x1 |- y3) {};
  \node[toplbl] at ($0.5*(x0|-y3)+0.5*(x1|-y3)$) {$\T[4,3,1]$};

  \draw[seg] (x2 |- y3) -- (x3 |- y3);
  \node[dot] at (x2 |- y3) {};
  \node[dot] at (x3 |- y3) {};
  \node[toplbl] at ($0.5*(x2|-y3)+0.5*(x3|-y3)$) {$\T[4,3,2]$};
  
  \draw[dashed, gray!40] ($(y0)+(-0.5,0.6)$) -- ($(y3)+(-0.5,-0.6)$);
  \draw[arrowstyle, cyan!100] ($(y0)+(-0.8,0)$) -- ($(y0)+(-0.8,-1.5)$) node[midway,left,rulelab] {Rule 1};
  \draw[arrowstyle, green!100] ($(y1)+(-0.8,0)$) -- ($(y1)+(-0.8,-1.5)$) node[midway,left,rulelab] {Rule 2};
  \draw[arrowstyle, orange!100] ($(y2)+(-0.8,0)$) -- ($(y2)+(-0.8,-1.5)$) node[midway,left,rulelab] {Rule 3};

\end{tikzpicture}
\end{center}
\end{remark}
\vskip+0.3cm
\begin{lemma}\label{lem:inclusion}
    $\A=\C=\F$.
\end{lemma}
\begin{proof}
\leavevmode
\\
By Remark \ref{rem:subdivision_rules}, if a segment $A_j^i$ is equal to some $\T[a_0,a_1,a_2,\dots a_n]$, then for every non-empty $\T[a_0,a_1,a_2,\dots a_n,k]$ (where $k\in\{1,2,3,4\}$), there will exist a descendant of $A_j^i$ such that it is equal to $\T[a_0,a_1,a_2,\dots a_n,k]$. Formally, if $\T[a_0,a_1,a_2,\dots a_n]= A_j^i\subseteq \A_i$ then for every non-empty $\T[a_0,\dots,a_n,k]$, there exists some $s,\nu$ such that $ \T[a_0,a_1,a_2,\dots a_n,k]=A_\nu^{i+s}\subseteq A_{i+s}\subseteq A_{i+1}$. Consequently, if $C_n\subseteq A_i$, then $C_{n+1}\subseteq A_{i+1}$. Since $A_0=\T[4,3]=C_2\subseteq A_0$, we obtain inductively that for every $\nu\in\Z_{\geq0}$, $C_{\nu+2}\subseteq A_\nu$. Letting $\nu \to \infty$ we get $\C\subseteq \A$.
\\
We must now show that $\A \subseteq \C$.
Recall that if a segment $A_j^i$ is equal to some $\T[a_0,\dots,a_n]$, then in at most 3 steps, all children of $A_j^i$ will be some subset of $C_{n+1}$. That is, $\left(A_j^i\cap A_{i+3}\right)\subseteq C_{n+1}$.\\
Now we will prove that if $A_i\subseteq C_n$, then $A_{i+3}\subseteq C_{n+1}$.\\
Assume $A_i\subseteq C_n$. Then, for all $j\in\{1,2,\dots,2^i\}$, we have $A_j^i\subseteq C_n$. Remark \ref{rem:subdivision_rules} implies that either $A_j^i$ itself, or an ancestor at most three generations above it, is equal to some $\T[a_0,\dots,a_n]\subseteq C_n$. Hence there exists some $s\in\{0,1,2,3\}$ and $\nu\in\Z_{\geq0}$ such that $A_\nu^{i-s}=\T[a_0,\dots,a_n]$. Using the conclusion obtained in the previous paragraph, we have $\left(A_\nu^{i-s}\cap A_{i-s+3}\right)\subseteq C_{n+1}$. Since $A_{i+3}\subseteq A_{i-s+3}$ and $A_j^i\subseteq A_\nu^{i-s}$, we further conclude that $\left(A_j^{i}\cap A_{i+3}\right)\subseteq \left(A_\nu^{i-s}\cap A_{i+3}\right)\subseteq \left(A_\nu^{i-s}\cap A_{i-s+3}\right)\subseteq C_{n+1}$. This means in at most 3 steps all children of $A_j^i$ will become subsets of $C_{n+1}$. 
Since $A_j^i$ was chosen arbitrarily, for all $j\in\{1,2,\dots2^i\}$ we have $(A_j^{i}\cap A_{i+3})\subseteq C_{n+1}$, thus
\[
A_{i+3}=A_{i}\cap A_{i+3}=\left(\bigcup_{j=1}^{2^i} A_j^i \right)\cap A_{i+3}=\bigcup_{j=1}^{2^i} \left(A_j^i\cap A_{i+3}\right)\subseteq C_{n+1}.
\]
Finally, we notice that $A_0=C_2\subseteq C_2$. Using the above observation, we obtain that for all $n\in\Z_{\geq0}$, $A_{3n}\subseteq C_{n+2}$. Now letting $n\to \infty$ we get $\A\subseteq \C$.
\end{proof}
\begin{corollary}
$\F$ is a Cantor-type set.
\end{corollary}
\begin{proof}
\leavevmode
Note that $\A$ is a Cantor-type set by definition, so $\F=\A$ has the same property.
\end{proof}
Our next goal is to prove that $\tau(\F) > 1$, to do this we use the following lemma.\\
\begin{lemma}
\label{lemma:comptrick}
Let $a,b,c,d$ be finite or infinite blocks of continued fractions and $n\in\Z_{\geq1}$.
\[
\begin{aligned}
&\text{If } 
           A_{2j-1}^{i} = [a_0; a_1, \dots, a_n, _{a}^{b}],\ 
           A_{2j}^{i} = [a_0; a_1, \dots, a_n, _{c}^{d}],\ 
            G_{j}^{i} = (\max A_{2j-1}^i,\min A_{2j}^i).\\ 
&\text{Then }
\max\left\{
\frac{\left|G_{j}^{i} \right|}{|A_{2j-1}^{i}|},\ 
\frac{\left|G_{j}^{i} \right|}{|A_{2j}^{i}|}
\right\} \leq
\max\left\{\frac{(\left[a\right] + 1)(\left[c\right] - [b])}{(\left[c\right] + 1)([b] - \left[a\right])}, 	\frac{(5[d]+1)(\left[c\right] - [b])}{(5[b]+1)([d] - \left[c\right])}
\right\}.
\end{aligned}
\]
\end{lemma}

\begin{proof}[Proof of Lemma \ref{lemma:comptrick}]
\mbox{}\\
We denote $\frac{p_k}{q_k} = [a_0; a_1, \dots, a_k]$. The recurrence relation of the sequence $\{q_k\}_{k_{\geq0}}$ is well-known: $q_{-1}=0,\ q_0=1$ and $q_{k}=a_{k}q_{k-1}+q_{k-2}$ for all $k\geq1$.\\
Let $\varepsilon_k = \frac{q_{k-1}}{q_k}$. Then we get $\varepsilon_0=0$ and $\varepsilon_k=\frac{1}{a_k+\frac{q_{k-2}}{q_{k-1}}}=\frac{1}{a_k+\varepsilon_{k-1}}$ for all $k\geq 1$.\\
Since
\[
\varepsilon_k=\frac{1}{a_k+\varepsilon_{k-1}}\leq\frac{1}{1+0}=1
,\,\,\,\,
\varepsilon_k=\frac{1}{a_k+\varepsilon_{k-1}}\geq\frac{1}{4+1   }=\frac15,
\]
we see that $\varepsilon_k\in[\frac15,1]$ for all $k\geq1$.\\
Because $\F\subset\T[4,3]$, every continued fraction involved in the process has $a_0,\dots,a_n$ at least of length 2, i.e. $n\geq1$.
\[
\begin{aligned}
|G_{j}^{i}| = \sup G_{j}^{i} - \inf G_{j}^{i}
= \left| \frac{p_n[c] + p_{n-1}}{q_n[c] + q_{n-1}} - \frac{p_n[b] + p_{n-1}}{q_n[b] + q_{n-1}} \right| = \frac{[c] - [b]}{q_n^2 ([c] + \varepsilon_n)([b] + \varepsilon_n)},
\end{aligned}
\]
Similarly, we get 
\[
\begin{aligned}
|A_{2j-1}^{i}| = \frac{[b] - \left[a\right]}{q_n^2 ([b] + \varepsilon_n)(\left[a\right] + \varepsilon_n)},\
|A_{2j}^{i}| = \frac{[d] - \left[c\right]}{q_n^2 ([d] + \varepsilon_n)(\left[c\right] + \varepsilon_n)}.
\end{aligned}
\]
By monotonicity, we deduce that 
\[
\begin{aligned}
\frac{\left|G_{j}^{i} \right|}{|A_{2j-1}^{i}|} 
&= \frac{([a] + \varepsilon_n)(\left[c\right] - [b])}{(\left[c\right] + \varepsilon_n)([b] - \left[a\right])} 
\overset{(\varepsilon_n\le 1)}{\leq} \frac{(\left[a\right] + 1)(\left[c\right] - [b])}{(\left[c\right] + 1)([b] - \left[a\right])},\\
\frac{\left|G_{j}^{i} \right|}{|A_{2j}^{i}|} 
&= \frac{([d] + \varepsilon_n)([c] - [b])}{([b] + \varepsilon_n)([d] - \left[c\right])} 
\overset{(\varepsilon_n\ge \frac15)}{\leq} \frac{(5[d]+1)(\left[c\right] - [b])}{(5[b]+1)([d] - \left[c\right])}.
\end{aligned}
\]
Thus,
\[
\max\left\{
\frac{\left|G_{j}^{i} \right|}{|A_{2j-1}^{i}|},\ 
\frac{\left|G_{j}^{i} \right|}{|A_{2j}^{i}|}
\right\} \leq\
\max\left\{	\frac{(\left[a\right] + 1)(\left[c\right] - [b])}{(\left[c\right] + 1)([b] - \left[a\right])}, \frac{(5[d]+1)(\left[c\right] - [b])}{(5[b]+1)([d] - \left[c\right])}
\right\}.
\]
\end{proof}
\mbox{}\\
We now compute a lower bound for the thickness of $\F$ using the lemma above, considering each type of segment.
\begin{enumerate}
    \item[Type 1:]
    \[
    \max\!\left\{
      \frac{|G_{j}^{i}|}{|A_{2j}^{i}|},\,
      \frac{|G_{j}^{i}|}{|A_{2j-1}^{i}|}
    \right\}
    \le
    \max\!\left\{
      \frac{-98845 + 678\sqrt{26565}}{168340},\;
      \frac{1714 + 16935\sqrt{26565}}{2895005}
    \right\}
    \le 0.964.
    \]

    \item[Type 2:]
    \[
    \max\!\left\{
      \frac{|G_{j}^{i}|}{|A_{2j}^{i}|},\,
      \frac{|G_{j}^{i}|}{|A_{2j-1}^{i}|}
    \right\}
    \le
    \max\!\left\{
      \frac{-3102208 + 28527\sqrt{26565}}{12625009},\;
      \frac{734627 + 22099\sqrt{26565}}{5347148}
    \right\}
    \le 0.811.
    \]

    \item[Type 3:]
    \[
    \max\!\left\{
      \frac{|G_{j}^{i}|}{|A_{2j}^{i}|},\,
      \frac{|G_{j}^{i}|}{|A_{2j-1}^{i}|}
    \right\}
    \le
    \max\!\left\{
      \frac{-327577 + 11883\sqrt{26565}}{3835324},\;
      \frac{142707 + 31409\sqrt{26565}}{5780868}
    \right\}
    \le 0.911.
    \]

    \item[Type 4:]
    \[
    \max\!\left\{
      \frac{|G_{j}^{i}|}{|A_{2j}^{i}|},\,
      \frac{|G_{j}^{i}|}{|A_{2j-1}^{i}|}
    \right\}
    \le
    \max\!\left\{
      \frac{734627 + 22099\sqrt{26565}}{5347148},\;
      \frac{871925 + 10719\sqrt{26565}}{26671920}
    \right\}
    \le 0.811.
    \]

    \item[Type 5:]
    \[
    \max\!\left\{
      \frac{|G_{j}^{i}|}{|A_{2j}^{i}|},\,
      \frac{|G_{j}^{i}|}{|A_{2j-1}^{i}|}
    \right\}
    \le
    \max\!\left\{
      \frac{142707 + 31409\sqrt{26565}}{5780868},\;
      \frac{47481125 + 242609\sqrt{26565}}{241501492}
    \right\}
    \le 0.911.
    \]

    \item[Type 6:]
    \[
    \max\!\left\{
      \frac{|G_{j}^{i}|}{|A_{2j}^{i}|},\,
      \frac{|G_{j}^{i}|}{|A_{2j-1}^{i}|}
    \right\}
    \le
    \max\!\left\{
      \frac{-1760165 + 12317\sqrt{26565}}{3740264},\;
      \frac{228339 + 83497\sqrt{26565}}{14071116}
    \right\}
    \le 0.984.
    \]

    \item[Type 7:]
    \[
    \max\!\left\{
      \frac{|G_{j}^{i}|}{|A_{2j}^{i}|},\,
      \frac{|G_{j}^{i}|}{|A_{2j-1}^{i}|}
    \right\}
    \le
    \max\!\left\{
      \begin{aligned}
      &\frac{124744719 + 3698485\sqrt{26565}}{897620716},\; \\
      &\frac{5862625949 + 36639108\sqrt{26565}}{113824090231}
      \end{aligned}
    \right\}
    \le 0.811.
    \]

    \item[Type 8:]
    \[
    \max\!\left\{
      \frac{|G_{j}^{i}|}{|A_{2j}^{i}|},\,
      \frac{|G_{j}^{i}|}{|A_{2j-1}^{i}|}
    \right\}
    \le
    \max\!\left\{
      \frac{24202879 + 5279855\sqrt{26565}}{973087996},\;
      \frac{61036589 + 376740\sqrt{26565}}{301409236}
    \right\}
    \le 0.910.
    \]

    \item[Type 9:]
    \[
    \max\!\left\{
      \frac{|G_{j}^{i}|}{|A_{2j}^{i}|},\,
      \frac{|G_{j}^{i}|}{|A_{2j-1}^{i}|}
    \right\}
    \le
    \max\!\left\{
      \begin{aligned}
      &\frac{-2678249986 + 26487231\sqrt{26565}}{13906866931},\; \\
      & \frac{7015092919 + 43130255\sqrt{26565}}{18089104276}
      \end{aligned}
    \right\}
    \le 0.777.
    \]
\end{enumerate}
So from the estimates above, we get  the inequality 
$$\max\!\left\{
      \frac{|G_{j}^{i}|}{|A_{2j}^{i}|},\,
      \frac{|G_{j}^{i}|}{|A_{2j-1}^{i}|}
    \right\}\le \frac{228339 + 83497\sqrt{26565}}{14071116}\approx 0.984.
    $$

\begin{corollary}
    \label{cor:F44}
    The set 
    $\F$ has thickness 
\begin{equation}\label{lambda}
   \tau(\F)\geq \frac{1}{\lambda}:=\frac{83497\sqrt{26565}-228339}{13158329} \approx 1.0169.      
\end{equation}

\end{corollary}
\begin{proof}
We can see that
$$\min\!\left\{
      \frac{|A_{2j}^{i}|}{|G_{j}^{i}|},\,
      \frac{|A_{2j-1}^{i}|}{|G_{j}^{i}|}
    \right\}=\frac{1}{\max\!\left\{
      \frac{|G_{j}^{i}|}{|A_{2j}^{i}|},\,
      \frac{|G_{j}^{i}|}{|A_{2j-1}^{i}|}
    \right\}}\geq\frac{1}{\frac{228339 + 83497\sqrt{26565}}{14071116}}=\frac{83497\sqrt{26565}-228339}{13158329}.$$
By definition, we have
$$\tau(\F)\geq \min\!\left\{
      \frac{|A_{2j}^{i}|}{|G_{j}^{i}|},\,
      \frac{|A_{2j-1}^{i}|}{|G_{j}^{i}|}
    \right\}\geq\frac{83497\sqrt{26565}-228339}{13158329},$$
thus proving the corollary.
\end{proof}
This completes the proof of Theorem \ref{thm:F4}.

\end{proof}
\begin{theorem}
    \label{thm:logF4}
    $
    \tau (\log \F)>1.
    $
\end{theorem}
In order to prove this statement, we use the following lemma.
\begin{lemma}\label{lem:log-interval}
Let
\[
J_1=(a,a+r),\qquad J_2=(a+r,a+r+s),\qquad J_3=(a+r+s,a+r+s+t)
\]
be some intervals
with \(a,r,s,t>0\). If
\begin{equation}\label{eq:s-condition}
s\le \min\!\left\{ r,\; \frac{2t}{\sqrt{1+\dfrac{4t}{a+r}}+1}\right\},
\end{equation}
then
\[
\lvert \log J_2\rvert \le \min\{\lvert \log J_1\rvert,\;\lvert \log J_3\rvert\}.
\]
\end{lemma}

\begin{proof}
First we prove \(\lvert\log J_2\rvert \le \lvert\log J_1\rvert\). This inequality is equivalent to
\[
\log(a+r+s)-\log(a+r)\le \log(a+r)-\log a,
\]
which is equivalent to
$as \le a r + r^2.$
This last inequality certainly holds when \(s\le r\), so the first part of \eqref{eq:s-condition} implies
\(\lvert\log J_2\rvert \le \lvert\log J_1\rvert\).
Next we prove \(\lvert\log J_2\rvert \le \lvert\log J_3\rvert\). The inequality
\[
\log(a+r+s)-\log(a+r)\le \log(a+r+s+t)-\log(a+r+s)
\]
is equivalent to the quadratic inequality
\[
s^2+(a+r)s-(a+r)t \le 0.
\]
Solving this quadratic inequality for \(s\ge 0\), we get the upper bound
\[
s \le \frac{a+r}{2}\Big(\sqrt{1+\tfrac{4t}{a+r}}-1\Big)
= \frac{2t}{\sqrt{1+\dfrac{4t}{a+r}}+1}.
\]
Thus the second part of \eqref{eq:s-condition} guarantees \(\lvert\log J_{2}\rvert \le \lvert\log J_3\rvert\).
Combining the two estimates, we obtain the desired conclusion.
\end{proof}

\begin{proof}[Proof of Theorem \ref{thm:logF4}]

By Lemma~\ref{lem:log-interval}, it remains only to verify that for any \(i,j\in\mathbb{N}\), \(1\le j\le 2^{i}\), the lengths
\[
r=r_{i,j}:=|A_{2j-1}^{i}|,\qquad s=s_{i,j}:=|A^{i-1}_j|-|A_{2j-1}^{i}|-|A_{2j}^{i}|,\qquad t=t_{i,j}:=|A_{2j}^{i}|
\]
of the subdivisions of the segment $A_0$ satisfy \eqref{eq:s-condition}, where
$a=a_{i,j}:=\min A_{2j-1}^{i}\ge [4;3,\overline{1,4,1,4,1,3}]$.
In the course of the proof of \autoref{thm:F4}, it was established that
\[
s \le \lambda\cdot\min\{r,t\}\leq\lambda\cdot t,\text{ where }
\lambda=\frac{228339 + 83497\sqrt{26565}}{14071116} \approx 0.9834.\  
\]
We wish to prove
\[
s \le\frac{2t}{\sqrt{1+\dfrac{4t}{a+r}}+1}.
\]
If
\[
\lambda\cdot t\leq \frac{2t}{\sqrt{1+\dfrac{4t}{a+r}}+1}
\]
the inequality holds immediately.\\
Hence, it is sufficient to verify the condition of Lemma~\ref{lem:log-interval} only in the possible cases when
\[
\frac{2t}{\sqrt{1+\dfrac{4t}{a+r}}+1} < \lambda\cdot t.\]
We prove that this case is not possible. The above inequality is obviously equivalent to:
\[\frac{t}{a+r} > \gamma\]
where 
\[\gamma=\frac{188261210808537 - 1136812239479\sqrt{26565}}{173141622072241}\approx0.017.
\]
In other words, it suffices to verify the inequality for those segments \(A_{2j}^{i}\) whose lengths satisfy
\[
t > \gamma\cdot (a+r) \ge \gamma\cdot a \geq \gamma\cdot[4;3,\overline{1,4,1,4,1,3}] > 0.07.
\]
But since each of the segments $A_{2j}^i$ is a subinterval of $\A_0$, 
and 
\[
|\A_0|=[4;3,\overline{4,1,4,1,3,1}]-[4;3,\overline{1,4,1,4,1,3}]=\frac{5501 - \sqrt{26565}}{1238}-\frac{783 + \sqrt{26565}}{222}\approx0.05,
\]
we conclude that no $A^i_{2j}$ will have length larger than 0.05, so this case cannot be achieved. Therefore $\tau(\log\F)>1.$
\end{proof}

\section{Proof of Theorem \ref{thm:main}}
\begin{proof}
We apply Corollary \ref{cor:logset} and Theorem \ref{thm:logF4} to obtain 
\[
\F\cdot\F=\left[[4;3,\overline{1,4,1,4,1,3}]^2,[4;3,\overline{4,1,4,1,3,1}]^2\right]=\left[\frac{106609+261\sqrt{26565}}{8214},\frac{15143783-5501\sqrt{26565}}{766322}\right].
\]
So
\[
\left[\mubound,10+6\sqrt{2}\right]\subset \F\cdot\F.
\]
\medskip
Therefore any number $\alpha\in\left[\mubound,10+6\sqrt{2}\right]$ can be represented in the form
\[
\alpha = [x_0;x_1,\dots,x_n,\dots]\cdot[y_0;y_1,\dots,y_n,\dots],
\]
where
$[x_0;x_1,\dots,x_n,\dots]\in\F$ and $[y_0;y_1,\dots,y_n,\dots]\in\F$.\\
Let
$\{n_i\}_{i\geq1}$ and $\{m_i\}_{i\geq1}$ be two strictly increasing sequences of natural numbers such that $x_{n_i}\neq4$ and $y_{m_i}\neq4$ for all $i\geq 1$ (by definition there will exist such sequences), and let
\[
S_i=(x_{n_i},x_{n_{i}+1},\dots,x_{1},x_0,y_0,y_1,\dots,y_{m_{i}-1},y_{m_i}),\,\,\,\,\,
x=[S_1,S_2,S_3\dots].
\]
Let \(k_i\) denote the index in the continued fraction of \(x\) that corresponds to the occurrence of the element \(x_0\) coming from the block \(S_i\). We observe that:
\[
\lim_{i\to\infty} \rho_{k_i}(x)
= \lim_{i\to\infty}\bigl[\,x_0;x_1,\dots,x_{n_i},S_i^*,S_{i-1}^*,\dots\,\bigr]\cdot
\bigl[\,y_0;y_1,\dots,y_{m_i},S_{i+1},\dots\,\bigr]=
\]
\[
= [x_0;x_1,\dots]\cdot[y_0;y_1,\dots] = \alpha,
\]
where \(S_i^*\) denotes $
[y_{m_i},y_{m_{i-1}},\dots,y_0,x_0,\dots,x_{n_{i-1}},x_{n_i}].
$
\\
We now demonstrate that the string $(4,1,4,1,4)$ cannot appear within the continued fraction expansion of $x$. By construction, this string is absent from the block $(x_0, x_1, \dots, x_{n_i})$. Because $(4,1,4,1,4)$ is palindromic, it must also be absent from the reversed block $(x_{n_i}, x_{n_i-1}, \dots, x_0)$. Similarly, it is excluded from $(y_0, y_1, \dots, y_{m_i})$ by definition. Furthermore, the string cannot bridge these adjacent blocks. It cannot span the junction $(x_0, y_0) = (4,4)$ because $(4,1,4,1,4)$ does not contain consecutive $4$s. Nor can it span the junction $(y_{m_i}, x_{n_{i+1}})$; the string $(4,1,4,1,4)$ strictly alternates $4$s with a single non-$4$ character, yet our initial assumptions dictate that both $y_{m_i} \neq 4$ and $x_{n_{i+1}} \neq 4$, meaning there would be consecutive non-$4$ characters at this boundary. Recall that $[\overline{4;1,4,1,3,1}]$ is the largest element in $\F$, and $[\overline{3;1,4,1,4,1}]$ is the largest element in $\F$ with a non-4 integer part. Since $\{k_i\}_{i\ge1}$ is the sequence indicating the indices of all appearances of $(4,4)$ in the continued fraction expansion of $x$, and $(4,1,4,1,4)$ never occurs in $x$, any sequence of indices $\{k'_i\}_{i\ge1}$ that shares at most finitely many terms with $\{k_i\}_{i\ge1}$ must satisfy one of the following conditions for all sufficiently large $i$:$$\alpha_{k'_i+1} \leq [\overline{4;1,4,1,3,1}] \quad \text{and} \quad \alpha^{*}_{k'_i} \leq [\overline{3;1,4,1,4,1}],$$or conversely,$$\alpha_{k'_i+1} \leq [\overline{3;1,4,1,4,1}] \quad \text{and} \quad \alpha^{*}_{k'_i} \leq [\overline{4;1,4,1,3,1}].$$
Consequently, we obtain$$\limsup_{i\to\infty} \rho_{k'_i}(x)\le [\overline{4;1,4,1,3,1}]\,[\overline{3;1,4,1,4,1}]=\mubound\le\alpha=\lim_{i\to\infty} \rho_{k_i}(x).$$It follows that$$\alpha=\lim_{i\to\infty}\rho_{k_i}(x)=\limsup_{n\to\infty}\rho_n(x)\in\dirich.$$Since $\alpha$ was chosen as an arbitrary element of $\left[\mubound,10+6\sqrt{2}\right]$, this completes the proof of Theorem \ref{thm:main}.\end{proof}

\section*{Acknowledgements}
The authors are grateful to Prof. Nikolay Moshchevitin, Dr. Yang Bohan, and Luo Zichen for their invaluable guidance, encouragement, and support throughout the preparation of this article.


\end{document}